\documentclass[11pt,a4paper]{amsart}

\usepackage[T1]{fontenc}
\usepackage{amsmath, amsfonts, amsthm}

\newtheorem{theorem}{Theorem}
\newtheorem{lemma}{Lemma}
\newtheorem{corollary}{Corollary}

\newcommand{\D}{\mathcal{D} (a, \gamma)}
\newcommand{\E}{\mathcal{E} (a, b, \gamma)}
\newcommand{\U}{\mathcal{U}_\varepsilon}
\renewcommand{\L}{\mathcal{L}_\varepsilon}

\allowdisplaybreaks[1]

\begin{document}

\title[Iterated function systems with smooth invariant densities]{Random iterated function systems with smooth invariant densities}
\author{Tomas Persson}
\address{Institute of Mathematics, Polish Academy of Sciences, ulica \'Sniadeckich~{8}, Warszawa, Poland}
\email{tomasp@impan.gov.pl}
\urladdr{www.impan.gov.pl/~tomasp/}

\begin{abstract}
  We consider some random iterated function systems on the interval and show that the invariant measure has density in $\mathcal{C}^\infty$. To prove this we use some techniques for contractions in cone metrics, applied to the transfer operator.
\end{abstract}

\maketitle

\section{Introduction}
Let $\{f_1, \ldots, f_n\}$ be an iterated function system on an interval $I$, and assume that at each iterate we choose a map according to the probability vector $(p_1, p_2, \ldots, p_n)$.
We restrict to the case when all the maps are contracting. By \cite{DiaconisFreedman} there is then a unique invariant probability measure. In the case when the images $f_i (I)$ have non empty intersection, it is often hard to determine the nature of the invariant measure. 

For example, the invariant measure of the random iterated function system $\{f_{-1}, f_1 \}$, with $f_i (x) = \lambda x + i (1-\lambda)$, $I = [-1, 1]$ and $\lambda \in (1/2, 1)$ is known to have an invariant measure that is absolutely continuous with respect to Lebesgue measure for almost all $\lambda \in (1/2, 1)$, see \cite{Solomyak}. 
The method used to prove this and similar results for almost all parameters is usually referred to as ''transversality''.
Only for a few specific values of $\lambda$ is it known whether the invariant measure is absolutely continuous och singular with respect to Lebesgue measure, see for instance \cite{Erdos} and \cite{Garsia}.

One way to get results for a specific parameter is to change the problem and instead consider the iterated function system with some random perturbation, see for instance \cite{PeresSimonSolomyak} and \cite{BaranyPersson}. It is then possible to apply the method of transversality to prove that the invariant measure is absolutely continuous with respect to Lebesgue measure, provided that some conditions are satisfied.
In fact more is proved. The conditional measures conditioned on the random perturbations have density in $L_2$, \cite{PeresSimonSolomyak}.

It is sometimes claimed that transversality is essentially the only known method to obtain the results described above, \cite{PeresSimonSolomyak}. In this paper we use an other method, originating from \cite{Liverani}, to show that a random iterated function system, consisting of affine maps with some small random perturbation, has an invariant measure which is absolutely continuous with respect to Lebesgue measure, and that the density is infinitely many times differentiable. Hence, this paper has two purposes --- to strengthen previous results and to introduce a new method for iterated function systems with overlaps.

\newpage

\section{A random iterated function system}

Let $I = [-1, 1]$. Take $\varepsilon \geq 0$ small (we will later let $\varepsilon > 0$). For $k=1, 2, \ldots, n$, we define functions
\[
  f_{k,t} \colon x \mapsto \lambda x + a_k + b_k t,
\]
where $\lambda$, $a_k$ and $b_k$ are numbers such that $|\lambda| < 1$, $b_k \ne 0$ and $f_{k,t} (I) \subset I$ for all $0 \leq t \leq \varepsilon$.

Let $(p_1, p_2, \ldots, p_n)$ be a probability vector. We consider the random iterated function system consisting of the maps $f_{k,t}$ applied at each iterate with probability $p_k$ and $t$ distributed according to a smooth function $h$, independently of previous iterates. We thus assume that $h$ is a non negative on $[0, \varepsilon]$, and that $\int_0^\varepsilon h (x) \, \mathrm{d}x = 1$. If $\varepsilon = 0$, this should be interpreted in the natural way.

By \cite{DiaconisFreedman} there is a unique invariant measure $\mu$ of this iterated function system. We are interested in the  properties of this measure. Let $\mathcal{C}^\infty (I)$ denote the set of functions on $I$, differentiable infinitely many times.
The following theorem will be proved.

\begin{theorem} \label{the:theorem}
  Let $\mu$ be the invariant measure of the random iterated function system $\{f_{1,t}, \ldots, f_{n,t} \}$ defined above, with corresponding probabilities $0 < p_k \leq 1$. Then $\mu$ has density $\phi$ in $\mathcal{C}^\infty (I)$, and
  \[
    \sup | \phi^{(k)} | \leq \lambda^{ - \frac{(k+1)(k+2) }{2} } \biggl( \sum_{i = 1}^n \frac{p_i}{b_i} (h(\varepsilon) + h(0) + \varepsilon \sup |h'|) \biggr)^{k+1}.
  \]
\end{theorem}

To prove this Theorem we will use of the method of contractions in cone metrics introduced by Liverani in \cite{Liverani}. This method was further developed for different hyperbolic dynamical systems by Viana in \cite{Viana}. The method of this paper is very much inspired by \cite{Viana}. However some differences are present, since we allow the images of the different function to overlap.
We use this method to construct the invariant measure as a fixed point of the transfer operator. This first part is valid for $\varepsilon \geq 0$. After this we assume that $\varepsilon > 0$ and prove that the density of the invariant measure is smooth.

In \cite{BaranyPersson}, the author and B. B\'ar\'any studied certain iterated function systems of similar type as in this paper, with $h(t) = 1/\varepsilon$. It was proved that the $L_2$-norm of the density does not grow faster than $1/\sqrt{\varepsilon}$ as $\varepsilon$ vanishes. For this case, the estimate in Theorem~\ref{the:theorem} provides us with $\sup \phi \leq c/\varepsilon$, where $c$ is some constant. This together with $\int_I \phi (x) \, \mathrm{d}x = 1$ imply that the $L_2$-norm does not grow faster than $1/\sqrt{\varepsilon}$, just as in \cite{BaranyPersson}.

\section{A transfer operator}

For $\varepsilon > 0$, we define the transfer operator $\L$ associated to the random iterated function system defined by
\begin{equation} \label{eq:Ldef}
  \L \phi (y) = \sum_{i = 1}^n \frac{p_i}{\lambda} \int_0^\varepsilon \phi \circ f_{i, t}^{-1} (y) \chi_{f_{i, t} (I)} (y) h(t) \, \mathrm{d}t,
\end{equation}
where $\phi \circ f_{i, t}^{-1} (y) \chi_{f_{i, t} (I)} (y)$ is defined to be zero if $y \not \in f_{i, t} (I)$. For $\varepsilon = 0$ the operator $\mathcal{L}_0$ is defined similarly, but without the integral over $t$.

We also introduce the operator $\U$ for $\varepsilon > 0$, defined by
\[
  \U \psi (x) = \sum_{i = 1}^n p_i \int_0^\varepsilon \psi \circ f_{i, t} (x) h (t) \, \mathrm{d}t.
\]
Similarly as above, we can also define the operator $\mathcal{U}_0$. For the proof of Theorem~\ref{the:theorem} we will not use the operators $\mathcal{L}_0$ and $\mathcal{U}_0$. However, most of what is done with the operators $\L$ and $\U$ is also valid for $\mathcal{L}_0$ and $\mathcal{U}_0$.

The operators $\L$ and $\U$ are related according to the following lemma. We introduce the symbol $m$ to denote the normalised Lebesgue measure on the interval $I$.

\begin{lemma} \label{lem:L-U}
  If $\phi$ and $\psi$ are continuous functions on $I$, then
  \[
    \int_I \psi \cdot \L \phi \, \mathrm{d}m = \int_I \U \psi \cdot \phi \, \mathrm{d}m. \qedhere
  \]
\end{lemma}

Lemma~\ref{lem:L-U} shows that $\L$ is related to invariant densities of the iterated function system in the following way. Suppose that we can define a measure $\mu$ as the weak limit of measures $\mu_{ n}$ defined by
\[
  \int_I \psi \, \mathrm{d} \mu_{ n} = \int_I \psi \cdot \L^n 1 \, \mathrm{d}m.
\]
Then $\mu (I) = 1$ and $\mu$ is invariant since
\[
  \mu (I) = \lim_{n \to \infty} \int_I 1 \cdot \L^n 1 \, \mathrm{d}m = \lim_{n \to \infty} \int_I \U^n 1 \, \mathrm{d}m = 1,
\]
and for any measurable $E \subset I$
\begin{multline*}
  \mu (E) = \lim_{n \to \infty} \int_I \chi_E \L^n 1 \, \mathrm{d}m = \lim_{n \to \infty} \int_I \U \chi_E \L^{n-1} 1 \, \mathrm{d}m \\
  = \lim_{n \to \infty} \int_I \sum_{i = 1}^n p_i \int_0^\varepsilon \chi_{f_{i, \varepsilon}^{-1} (E)} h(t) \, \mathrm{d}t \, \L^{n-1} 1 \, \mathrm{d}m \\
  = \sum_{i = 1}^n p_i \int_0^\varepsilon \mu (f_{i, t}^{-1} (E)) h(t) \, \mathrm{d}t.
\end{multline*}

\begin{proof}[Proof of Lemma~\ref{lem:L-U}]
  The lemma is a simple consequence of the formula for change of variable and Fubini's theorem. Indeed,
  \begin{align*}
    \int_I \psi \cdot \L \phi \, \mathrm{d} m &= \int_I \psi (y) \sum_{i = 1}^n \frac{p_i}{\lambda} \int_0^\varepsilon \phi \circ f_{i, t}^{-1} (y) \chi_{f_{i, t} (I)} (y) h(t) \, \mathrm{d}t \mathrm{d}y = \\
    &= \int_0^\varepsilon \sum_{i = 1}^n \int_I \frac{p_i}{\lambda} \psi (y) \phi \circ f_{i, t}^{-1} (y) \chi_{f_{i, t} (I)} (y) h (t) \, \mathrm{d}y \mathrm{d}t = \\
    &= \int_0^\varepsilon \sum_{i = 1}^n \int_I p_i \psi \circ f_{i, t} (x) \phi (x) h(t) \, \mathrm{d}x \mathrm{d}t 
    = \int_I \U \psi \cdot \phi \, \mathrm{d}m.
  \end{align*}
\end{proof}

\section{Cones and the Hilbert metric}

For the theory in this section we refer to \cite{Birkhoff}, \cite{Viana} and \cite{Liverani}.

Let $E$ be a vector space. (We will later let $E$ be function spaces on $I$.) A convex cone $C \subset E$ is a set such that
\[
  v_1, v_2 \in C \ \mathrm{and} \ t_1, t_2 > 0 \quad \Rightarrow \quad t_1 v_1 + t_2 v_2 \in C.
\]
Assume that $(- \overline{C}) \cap \overline{C} = \{0 \}$. Define
\begin{align*}
  \alpha (v_1, v_2) &= \sup \{ \, t > 0 : v_2 - t v_1 \in C \, \}, \\
  \beta (v_1, v_2) &= \inf \{ \, s > 0 : s v_1 - v_2 \in C \, \}.
\end{align*}
Let
\[
  \theta_C (v_1, v_2) = \log \frac{\beta (v_1, v_2) }{ \alpha (v_1, v_2) }.
\]
Then $\theta_C$ is a metric on the quotient $C / \sim$, where $\sim$ is the equivalent relation $v_1 \sim v_2$ if and only if there exists a $t > 0$ such that $v_1 = t v_2$.
The metric $\theta_C$ is called the projective metric or the Hilbert metric of the cone $C$.

We have the following remarkable theorem by G. Birkhoff.

\begin{theorem} \label{the:contraction}
  Let $\mathcal{L}$ be a linear operator and $C_1$ and $C_2$ two convex cones with $(-C_1) \cap C_1 = \{0\}$ and $(-C_2) \cap C_2 = \{0\}$, such that $\mathcal{L} (C_1) \subset C_2$. If $D = \sup \{\, \theta_{C_2} (\mathcal{L} (v_1), \mathcal{L} (v_2)) : v_1, v_2 \in C_1 \, \}$ is finite then
  \[
    \theta_{C_2} (\mathcal{L} (v_1), \mathcal{L} (v_2)) \leq \tanh \Bigl( \frac{D}{4} \Bigr) \theta_{C_1} (v_1, v_2)
  \]
  for all $v_1, v_2 \in C_1$.
\end{theorem}

For a proof of this theorem, see \cite{Birkhoff}, \cite{Viana} and/or \cite{Liverani}.

We note that $C / \sim$ endowed with the projective metric need not be a complete space, so the limit $\lim_{n \to \infty} \mathcal{L}^n v$ may not exist, although the sequence $( \mathcal{L}^n v)_{n = 0}^\infty$ is Cauchy by Theorem~\ref{the:contraction}.

\section{The operators $\L$ and $\U$}

We let $\mathcal{C}^n (I)$ be the set of $n$ times continuously differentiable functions on $I$, and $\mathcal{C}_0^n (I)$ be the set of functions in $\mathcal{C}^n (I)$ such that the derivative of order $k$ vanishes at the endpoints of $I$ for $0 \leq k \leq n$. One observes immediately that for $\varepsilon > 0$
\[
  \L \colon \mathcal{C}^{n-1} (I) \to \mathcal{C}_0^{n} (I), \quad n \in \mathbb{N},
\]
and
\[
  \U \colon \mathcal{C}^{n-1} (I) \to \mathcal{C}^{n} (I), \quad n \in \mathbb{N}.
\]

Moreover, if $\phi \in \mathcal{C}_0^1 (I)$ and $\psi \in \mathcal{C}^1 (I)$ then
\[
  \frac{\mathrm{d}}{\mathrm{d} y} \L \phi (y) = \frac{1}{\lambda} \L \phi' (y) \quad \mathrm{and} \quad \frac{\mathrm{d}}{\mathrm{d} x} \U \psi (x) = \lambda \U \psi' (x).
\]

These properties will be used to prove that the invariant measure has density in $\mathcal{C}^\infty$. Before we do this we must first prove that the measures $\mu_{ n}$ defined above, converges weakly to some measure $\mu$.

\section{Cones}

We will introduce a cone of densities on which $\L$ operates, and construct the measure $\mu$ mentioned above as a weak limit of $\L^n$ operating on these densities.
This constructions follows very closely to the construction in chapter~4 of \cite{Viana}. We do however include all proofs for completeness. There are some differences, in particular in the proof of Lemma~\ref{lem:finiteEdiam}, due to the overlapping images.

This first part of the construction is valid also for $\varepsilon = 0$. After constructing $\mu$ as a weak limit, we restrict to the case of positive $\varepsilon$ and show that $\mu$ is absolutely continuous with respect to Lebesgue measure, with density in $\mathcal{C}^\infty$.

We introduce the cone
\[¨
  \D = \{\, \rho \in \mathcal{C} (I) : \rho > 0 \ \mathrm{and} \ \log \rho \ \mathrm{is} \ (a, \gamma)\text{-H\"older} \,\}.
\]
That a function $\phi$ is $(a, \gamma)$-H\"older, means that $| \phi (x) - \phi(y) | \leq a |x - y|^\gamma$, for all $x$ and $y$.

Let us find the cone metric of $\D$. Let $\rho_1, \rho_1 \in \D$ and $t > 0$. By
\[
  \rho_2 - t \rho_1 > 0 \quad \Leftrightarrow \quad t < \frac{\rho_2}{\rho_1},
\]
\[
  \frac{ (\rho_2 - t \rho_1) (x)}{ (\rho_2 - t \rho_1 ) (y) } \leq e^{a |x - y|^\gamma} \quad \Leftrightarrow \quad t \leq \frac{ e^{a |x - y|^\gamma} \rho_2 (y) - \rho_2 (x) }{ e^{a |x - y|^\gamma} \rho_1 (y) - \rho_1 (x) }
\]
and
\[
  \frac{ (\rho_2 - t \rho_1) (x) }{ (\rho_2 - t \rho_1) (y) } \geq e^{ - a |x - y|^\gamma } \quad \Leftrightarrow \quad t \leq \frac{ e^{a |x - y|^\gamma} \rho_2 (x) - \rho_2 (y) }{ e^{a |x - y|^\gamma} \rho_1 (x) - \rho_1 (y) },
\]
we get that
\[
  \alpha_\mathcal{D} (\rho_1, \rho_2) = \inf_{x \ne y} \biggl\{ \frac{\rho_2 (x)}{\rho_1 (x)}, \frac{ e^{a |x - y|^\gamma} \rho_2 (y) - \rho_2 (x) }{ e^{a |x - y|^\gamma} \rho_1 (y) - \rho_1 (x) } \biggr\}.
\]
Similarly we get that
\[
  \beta_\mathcal{D} (\rho_1, \rho_2) = \sup_{x \ne y} \biggl\{ \frac{\rho_2 (x)}{\rho_1 (x)}, \frac{ e^{a |x - y|^\gamma} \rho_2 (y) - \rho_2 (x) }{ e^{a |x - y|^\gamma} \rho_1 (y) - \rho_1 (x) } \biggr\}.
\]
Hence
\[
  \theta_{\mathcal{D}} (\rho_1, \rho_2) = \log \frac{ { \displaystyle \sup_{x \ne y} } \biggl\{ \frac{\rho_2 (x)}{\rho_1 (x)}, \frac{ e^{a |x - y|^\gamma} \rho_2 (y) - \rho_2 (x) }{ e^{a |x - y|^\gamma} \rho_1 (y) - \rho_1 (x) } \biggr\} }{ { \displaystyle \inf_{x \ne y} } \biggl\{ \frac{\rho_2 (x)}{\rho_1 (x)}, \frac{ e^{a |x - y|^\gamma} \rho_2 (y) - \rho_2 (x) }{ e^{a |x - y|^\gamma} \rho_1 (y) - \rho_1 (x) } \biggr\} }
\]
is the projective metric on $\D$.

\begin{lemma} \label{lem:contractiononD}
  There exists a number $\lambda_0 < 1$ such that
  \begin{itemize}
    \item[{\em i)}] $\U \colon \mathcal{D} (a \gamma) \to \mathcal{D} (\lambda^\gamma a, \gamma)$,

    \item[{\em ii)}] if $\rho_1, \rho_2 \in \D$ then $\theta_\mathcal{D} (\U \rho_1, \U \rho_2) \leq \lambda_0 \theta_\mathcal{D} (\rho_1, \rho_2)$.
  \end{itemize}
\end{lemma}

\begin{proof}
  \noindent {\em i)} \quad We show that if $\rho \in \D$, then $\rho \circ f_i \in \mathcal{D} (\lambda^\gamma a, \gamma)$. Since $\mathcal{D} (\lambda^\gamma a, \gamma)$ is a cone, this implies that $\U \rho \in \mathcal{D} (\lambda^\gamma a, \gamma)$. Clearly,
  \[
    | \log \rho_1 ( f_i (x) ) - \log \rho_1 ( f_i (y) ) | \leq a | f_i (x) - f_i(y) |^\gamma = \lambda^\gamma a | x - y |^\gamma. 
  \]
  It is also clear that $\rho \circ f_i > 0$. This proves that $\rho \circ f_i \in \mathcal{D} (\lambda^\gamma a, \gamma)$.

  {\em ii)} \quad By Theorem~\ref{the:contraction} it is sufficient to show that $\mathcal{D} (\lambda a, \gamma)$ has finite diameter in $\D$, in order to conclude statement {\em ii)}. Let $\rho_1, \rho_2 \in \mathcal{D} (\lambda a, \gamma)$. Then
  \begin{multline*}
    \frac{ e^{a |x - y|^\gamma} \rho_2 (y) - \rho_2 (x) }{ e^{a |x - y|^\gamma} \rho_1 (y) - \rho_1 (x) } \geq \frac{\rho_2 (y)}{\rho_1 (y)} \frac{ e^{a |x - y|^\gamma} - e^{ \lambda a |x - y|^\gamma } }{ e^{a |x - y|^\gamma} - e^{ -\lambda a |x - y|^\gamma } } \\
    \geq \frac{\rho_2 (y)}{\rho_1 (y)} \inf_{z > 1} \frac{ z - z^\lambda }{ z - z^{-\lambda} } = \frac{\rho_2 (y)}{\rho_1 (y)} \frac{1 - \lambda}{1 + \lambda}.
  \end{multline*}
  Similarly we get
  \[
    \frac{ e^{a |x - y|^\gamma} \rho_2 (y) - \rho_2 (x) }{ e^{a |x - y|^\gamma} \rho_1 (y) - \rho_1 (x) } \leq \frac{\rho_2 (y)}{\rho_1 (y)} \frac{1 + \lambda}{1 - \lambda}.
  \]
  This implies that
  \begin{align*}
    \theta_{\mathcal{D}} (\rho_1, \rho_2) &\leq 2 \log \frac{1 + \lambda}{1 - \lambda} + \log \sup_{x \in I} \frac{\rho_2 (x)}{\rho_1 (x)} - \log \inf_{x \in I} \frac{\rho_2 (x)}{\rho_1 (x)} \\
    &\leq 2 \log \frac{1 + \lambda}{1 - \lambda} + 2 \lambda a |I|^\gamma = 2 \log \frac{1 + \lambda}{1 - \lambda} + 2^{1+\gamma} \lambda a.
  \end{align*}
  By Theorem~\ref{the:contraction} we can take $\lambda_0 = \tanh \bigl( \frac{1}{2} \log \frac{1 + \lambda}{1 - \lambda} + 2^{\gamma-1} \lambda a \bigr)$.
\end{proof}

We will now use the cone $\D$ to define a cone $\E$ on which we will let $\L$ operate. Let
\begin{multline*}
  \E = \biggl\{\, \phi \mathrm{\ is\ bounded} : \int_I \phi \rho \, \mathrm{d}m > 0 \ \mathrm{for} \ \mathrm{all} \ \rho \in \D, \\ \mathrm{and} \ \biggl| \log \frac{ \int_I \phi \rho_1 \, \mathrm{d}m }{ \int_I \rho_1 \, \mathrm{d}m } - \log \frac{ \int_I \phi \rho_2 \, \mathrm{d}m }{ \int_I \rho_2 \, \mathrm{d}m } \biggr| < b \theta(\rho_1, \rho_2) \, \biggr\}.
\end{multline*}
Note that $\phi \in \E$ does not need to be positive.
The set $\E$ is a convex cone according to the following lemma.

\begin{lemma}
  $\E$ is a convex cone and $ (- \overline{\E} ) \cap \overline{\E} = \{ 0 \}$.
\end{lemma}

\begin{proof}
  If $\alpha, \beta$ are positive real numbers, $\phi_1, \phi_2 \in \E$, and $\rho \in \D$ then
  \[
    \int (\alpha \phi_1 + \beta \phi_2) \rho \, \mathrm{d}m = \alpha \int \phi_1 \rho \, \mathrm{d}m + \beta \int \phi_2 \rho \, \mathrm{d}m > 0.
  \]

  Assume that $A, a_1, a_2, b_1, b_2$ are positive numbers such that
  \[
    e^{-A} \leq \frac{a_i}{b_i} \leq e^A, \quad i=1,2.
  \]
  Assume that $a_1 b_2 \geq a_2 b_1$ then
  \[
    \frac{a_1 + a_2}{b_1 + b_2} \frac{b_1}{a_1} = \frac{a_1 b_1 + a_2 b_1}{a_1 b_1 + a_1 b_2} \leq 1 \ \Rightarrow \ \frac{a_1 + a_2}{b_1 + b_2} \leq \frac{a_1}{b_1} \leq e^A,
  \]
  and
  \[
    \frac{a_1 + a_2}{b_1 + b_2} \frac{b_2}{a_2} = \frac{a_1 b_2 + a_2 b_2}{a_2 b_1 + a_2 b_2} \geq 1 \ \Rightarrow \ \frac{a_1 + a_2}{b_1 + b_2} \geq \frac{a_2}{b_2} \geq e^{-A}.
  \]
  This shows that $\{\, (a,b) : |\log a - \log b| \leq A \,\}$ is a convex cone. Hence, if $\phi_1, \phi_2 \in \E$, and $\rho_1, \rho_2 \in \D (a, \gamma)$ then
  \[
    \biggl| \log \frac{ \int_I (\alpha \phi_1 + \beta \phi_2) \rho_1 \, \mathrm{d}m }{ \int_I \rho_1 \, \mathrm{d}m } - \log \frac{ \int_I (\alpha \phi_1 + \beta \phi_2) \rho_2 \, \mathrm{d}m }{ \int_I \rho_2 \, \mathrm{d}m } \biggr| \leq b \theta_\mathcal{D} (\rho_1, \rho_2).
  \]
  This shows that $\E$ is a convex cone.

  To show that $ (- \overline{\E} ) \cap \overline{\E} = \{ 0 \}$, we will prove that if a bounded $\phi$ is such that $\int_I \phi \rho \, \mathrm{d}m = 0$ for all $\rho \in \D$, then $\phi = 0$.

  Let $\psi$ be $\gamma$-H\"older on $I$. Then
  \[
    \psi = (\psi^+ + B) - (\psi^- + B),
  \]
  where $B$ is a number, and $\psi^+$ and $\psi^-$ are non negative functions such that $\psi = \psi^+ - \psi^-$. If $B$ is sufficiently large, then $\psi^+ + B$ and $\psi^- + B$ are both in $\D$. Hence $\int_I \phi \psi \, \mathrm{d}m = 0$ for all $\gamma$-H\"older $\psi$.
  Since $\mathcal{C} (I)$ is dense in $L_1 (I)$, and $\mathcal{C} (I)$ is contained in the set of $\gamma$-H\"older functions on $I$, we conclude that $\int_I \phi \psi \, \mathrm{d}m = 0$ for all bounded $\psi$. Taking $\psi = \phi$ we conclude that $\phi = 0$ a.e.
\end{proof}

We will find a formula for the Hilbert metric of the cone $\E$. Recall that
\begin{align*}
  \alpha (\phi_1, \phi_2) &= \sup \{ \, t > 0 : \phi_2 - t \phi_1 \in \E \, \}, \\
  \beta (\phi_1, \phi_2) &= \inf \{ \, s > 0 : s \phi_1 - \phi_2 \in \E \, \}.
\end{align*}

Assume that $\int_I \rho_1 \, \mathrm{d}m = \int_I \rho_2 \, \mathrm{d}m = 1$. By
\[
  \int_I (\phi_2 - t \phi_1) \rho \, \mathrm{d}m > 0 \quad \Leftrightarrow \quad t < \frac{ \int_I \phi_2 \rho \, \mathrm{d}m }{ \int_I \phi_1 \rho \, \mathrm{d}m }
\]
and
\begin{align*}
  &\frac{ \int_I (\phi_2 - t \phi_1) \rho_1 \, \mathrm{d}m }{ \int_I (\phi_2 - t \phi_1) \rho_2 \, \mathrm{d}m } \leq e^{b \theta(\rho_1, \rho_2)}
  &\Leftrightarrow &
  &t \leq \frac{ \int_I \phi_2 \rho_2 \, \mathrm{d}m }{ \int_I \phi_1 \rho_2 \, \mathrm{d}m } \frac{ e^{b \theta(\rho_1, \rho_2)} - \frac{ \int \phi_2 \rho_1 \, \mathrm{d}m }{ \int_I \phi_2 \rho_2 \, \mathrm{d}m } }{  e^{b \theta(\rho_1, \rho_2)} - \frac{ \int_I \phi_1 \rho_1 \, \mathrm{d}m }{ \int_I \phi_1 \rho_2 \, \mathrm{d}m } } \\
  &\frac{ \int_I (\phi_2 - t \phi_1) \rho_1 \, \mathrm{d}m }{ \int_I (\phi_2 - t \phi_1) \rho_2 \, \mathrm{d}m } \geq e^{ - b \theta(\rho_1, \rho_2)} 
  &\Leftrightarrow &
  &t \leq \frac{ \int_I \phi_2 \rho_1 \, \mathrm{d}m }{ \int_I \phi_1 \rho_1 \, \mathrm{d}m } \frac{ e^{b \theta(\rho_1, \rho_2)} - \frac{ \int \phi_2 \rho_2 \, \mathrm{d}m }{ \int_I \phi_2 \rho_1 \, \mathrm{d}m } }{  e^{b \theta(\rho_1, \rho_2)} - \frac{ \int_I \phi_1 \rho_2 \, \mathrm{d}m }{ \int_I \phi_1 \rho_1 \, \mathrm{d}m } }
\end{align*}
we get that
\begin{multline*}
  \alpha_\mathcal{E} (\phi_1, \phi_2) = \inf \Biggl\{ 
  \frac{ \int_I \phi_2 \rho_1 \, \mathrm{d}m }{ \int_I \phi_1 \rho_1 \, \mathrm{d}m },
  \frac{ \int_I \phi_2 \rho_2 \, \mathrm{d}m }{ \int_I \phi_1 \rho_2 \, \mathrm{d}m } \frac{ e^{b \theta(\rho_1, \rho_2)} - \frac{ \int \phi_2 \rho_1 \, \mathrm{d}m }{ \int_I \phi_2 \rho_2 \, \mathrm{d}m } }{  e^{b \theta(\rho_1, \rho_2)} - \frac{ \int_I \phi_1 \rho_1 \, \mathrm{d}m }{ \int_I \phi_1 \rho_2 \, \mathrm{d}m } }, \\
  \frac{ \int_I \phi_2 \rho_1 \, \mathrm{d}m }{ \int_I \phi_1 \rho_1 \, \mathrm{d}m } \frac{ e^{b \theta(\rho_1, \rho_2)} - \frac{ \int \phi_2 \rho_2 \, \mathrm{d}m }{ \int_I \phi_2 \rho_1 \, \mathrm{d}m } }{  e^{b \theta(\rho_1, \rho_2)} - \frac{ \int_I \phi_1 \rho_2 \, \mathrm{d}m }{ \int_I \phi_1 \rho_1 \, \mathrm{d}m } }
  \Biggr\},
\end{multline*}
where the infimum is over all $\rho_1$ and $\rho_2$ with $\int_I \rho_1 \, \mathrm{d}m = \int_I \rho_2 \, \mathrm{d}m = 1$.

Similarly we get
\[
  \int_I (s \phi_2 - \phi_1) \rho \, \mathrm{d}m > 0 \quad \Leftrightarrow \quad s > \frac{ \int_I \phi_1 \rho \, \mathrm{d}m }{ \int_I \phi_2 \rho \, \mathrm{d}m }
\]
and
\begin{align*}
  &\frac{ \int_I (s \phi_2 - \phi_1) \rho_1 \, \mathrm{d}m }{ \int_I (s \phi_2 - \phi_1) \rho_2 \, \mathrm{d}m } \leq e^{b \theta(\rho_1, \rho_2)}
  &\Leftrightarrow &
  &s \geq \frac{ \int_I \phi_1 \rho_2 \, \mathrm{d}m }{ \int_I \phi_2 \rho_2 \, \mathrm{d}m } \frac{ e^{b \theta(\rho_1, \rho_2)} - \frac{ \int \phi_1 \rho_1 \, \mathrm{d}m }{ \int_I \phi_1 \rho_2 \, \mathrm{d}m } }{  e^{b \theta(\rho_1, \rho_2)} - \frac{ \int_I \phi_2 \rho_1 \, \mathrm{d}m }{ \int_I \phi_2 \rho_2 \, \mathrm{d}m } } \\
  &\frac{ \int_I (s \phi_2 - \phi_1) \rho_1 \, \mathrm{d}m }{ \int_I (s \phi_2 - \phi_1) \rho_2 \, \mathrm{d}m } \geq e^{ - b \theta(\rho_1, \rho_2)}
  &\Leftrightarrow &
  &s \geq \frac{ \int_I \phi_1 \rho_1 \, \mathrm{d}m }{ \int_I \phi_2 \rho_1 \, \mathrm{d}m } \frac{ e^{b \theta(\rho_1, \rho_2)} - \frac{ \int \phi_1 \rho_2 \, \mathrm{d}m }{ \int_I \phi_1 \rho_1 \, \mathrm{d}m } }{  e^{b \theta(\rho_1, \rho_2)} - \frac{ \int_I \phi_2 \rho_2 \, \mathrm{d}m }{ \int_I \phi_2 \rho_1 \, \mathrm{d}m } }
\end{align*}
Hence
\begin{multline*}
  \beta_\mathcal{E} (\phi_1, \phi_2) = \sup \Biggl\{ 
  \frac{ \int_I \phi_1 \rho_1 \, \mathrm{d}m }{ \int_I \phi_2 \rho_1 \, \mathrm{d}m },
  \frac{ \int_I \phi_1 \rho_2 \, \mathrm{d}m }{ \int_I \phi_2 \rho_2 \, \mathrm{d}m } \frac{ e^{b \theta(\rho_1, \rho_2)} - \frac{ \int \phi_1 \rho_1 \, \mathrm{d}m }{ \int_I \phi_1 \rho_2 \, \mathrm{d}m } }{  e^{b \theta(\rho_1, \rho_2)} - \frac{ \int_I \phi_2 \rho_1 \, \mathrm{d}m }{ \int_I \phi_2 \rho_2 \, \mathrm{d}m } }, \\
  \frac{ \int_I \phi_1 \rho_1 \, \mathrm{d}m }{ \int_I \phi_2 \rho_1 \, \mathrm{d}m } \frac{ e^{b \theta(\rho_1, \rho_2)} - \frac{ \int \phi_1 \rho_2 \, \mathrm{d}m }{ \int_I \phi_1 \rho_1 \, \mathrm{d}m } }{  e^{b \theta(\rho_1, \rho_2)} - \frac{ \int_I \phi_2 \rho_2 \, \mathrm{d}m }{ \int_I \phi_2 \rho_1 \, \mathrm{d}m } }
  \Biggr\},
\end{multline*}
where the supremum is over all $\rho_1$ and $\rho_2$ with $\int_I \rho_1 \, \mathrm{d}m = \int_I \rho_2 \, \mathrm{d}m = 1$.
We  now get $\theta_\mathcal{E}$ as $\theta_\mathcal{E} = \log \frac{\beta_\mathcal{E}}{\alpha_\mathcal{E}}$, but let us not write down the exact formula, as it would consume quite some space.

We will prove the following lemma in a similar way as we did for Lemma~\ref{lem:contractiononD}. While doing this, we will make use of Lemma~\ref{lem:contractiononD}.

\begin{lemma} \label{lem:finiteEdiam}
  The diameter of $\L (\E)$ is finite in $\E$, provided that $b > \Bigl(1 - \tanh \bigl( \frac{1}{2} \log \frac{1 + \lambda}{1 - \lambda} + 2^{\gamma-1} \lambda a \bigr) \Bigr)^{-1}$.
\end{lemma}

\begin{proof}
  Let $\rho_1 , \rho_2 \in \D$ and $\phi \in \E$. Then
  \begin{align*}
    \biggl| \log & \frac{ \int_I \rho_1 \L \phi \, \mathrm{d}m }{ \int_I \rho_1 \, \mathrm{d}m } - \log \frac{ \int_I \rho_2 \L \phi \, \mathrm{d}m }{ \int_I \rho_2 \, \mathrm{d}m } \biggr| \\
    & = \biggl| \log \frac{ \int_I \U \rho_1 \phi \, \mathrm{d}m }{ \int_I \rho_1 \, \mathrm{d}m } - \log \frac{ \int_I \U \rho_2 \phi \, \mathrm{d}m }{ \int_I \rho_2 \, \mathrm{d}m } \biggr| \\
    & \leq \biggl| \log \frac{ \int_I \U \rho_1 \phi \, \mathrm{d}m }{ \int_I \U \rho_1 \, \mathrm{d}m } - \log \frac{ \int_I \U \rho_2 \phi \, \mathrm{d}m }{ \int_I \U \rho_2 \, \mathrm{d}m } \biggr| \\
    & \hspace{4cm} + \biggl| \log \frac{ \int_I \U \rho_1 \, \mathrm{d}m }{ \int_I \rho_1 \, \mathrm{d}m } - \log \frac{ \int_I \U \rho_2 \, \mathrm{d}m }{ \int_I \rho_2 \, \mathrm{d}m } \biggr| \\
    & = \biggl| \log \frac{ \int_I \U \rho_1 \phi \, \mathrm{d}m }{ \int_I \U \rho_1 \, \mathrm{d}m } - \log \frac{ \int_I \U \rho_2 \L \phi \, \mathrm{d}m }{ \int_I \U \rho_2 \, \mathrm{d}m } \biggr| + \Biggl| \log \frac{ \frac{ \int_I \U \rho_1 \, \mathrm{d}m }{ \int_I \U \rho_2 \, \mathrm{d}m } }{ \frac{ \int_I \rho_1 \, \mathrm{d}m }{ \int_I \rho_2 \, \mathrm{d}m } } \Biggr|.
  \end{align*}
  We now observe that
  \[
    \sup \biggl\{ \frac{\psi_1}{\psi_2} \biggr\} \geq \frac{\psi_1 (x)}{\psi_2 (x)} \geq \inf \biggl\{ \frac{\psi_1}{\psi_2} \biggr\} \quad \Rightarrow \quad \sup \biggl\{ \frac{\psi_1}{\psi_2} \biggr\} \geq \frac{ \int_I \psi_1 \, \mathrm{d}m }{ \int_I \psi_2 \, \mathrm{d}m} \geq \inf \biggl\{ \frac{\psi_1}{\psi_2} \biggr\}
  \]
  holds for all positive functions $\psi_1$ and $\psi_2$. This implies that
  \[
    \frac{ \inf \Bigl\{ \frac{\rho_1}{\rho_2} \Bigr\} }{ \sup \Bigl\{ \frac{\rho_1}{\rho_2} \Bigr\} }
    \leq 
    \frac{ \inf \Bigl\{ \frac{\U \rho_1}{\U \rho_2} \Bigr\} }{ \sup \Bigl\{ \frac{\rho_1}{\rho_2} \Bigr\} }
    \leq 
    \frac{ \frac{ \int_I \U \rho_1 \, \mathrm{d}m }{ \int_I \U \rho_2 \, \mathrm{d}m } }{ \frac{ \int_I \rho_1 \, \mathrm{d}m }{ \int_I \rho_2 \, \mathrm{d}m } }
    \leq
    \frac{ \sup \Bigl\{ \frac{\U \rho_1}{\U \rho_2} \Bigr\} }{ \inf \Bigl\{ \frac{\rho_1}{\rho_2} \Bigr\} }
    \leq
    \frac{ \sup \Bigl\{ \frac{\rho_1}{\rho_2} \Bigr\} }{ \inf \Bigl\{ \frac{\rho_1}{\rho_2} \Bigr\} }.
  \]
  It follows that
  \[
    \Biggl| \log \frac{ \frac{ \int_I \U \rho_1 \, \mathrm{d}m }{ \int_I \U \rho_2 \, \mathrm{d}m } }{ \frac{ \int_I \rho_1 \, \mathrm{d}m }{ \int_I \rho_2 \, \mathrm{d}m } } \Biggr|
    \leq
    \theta_\mathcal{D} (\rho_1, \rho_2).
  \]
  Using this, we now get that
  \begin{multline*}
    \biggl| \log \frac{ \int_I \rho_1 \L \phi \, \mathrm{d}m }{ \int_I \rho_1 \, \mathrm{d}m } - \log \frac{ \int_I \rho_2 \L \phi \, \mathrm{d}m }{ \int_I \rho_2 \, \mathrm{d}m } \biggr|
    \leq b \theta_\mathcal{D} (\U \rho_1, \U \rho_2) + \theta_\mathcal{D} (\rho_1, \rho_2) \\
    \leq (b \lambda_0 + 1) \theta_\mathcal{D} (\rho_1, \rho_2),
  \end{multline*}
  where we used Lemma~\ref{lem:contractiononD} in the last step.
  Hence if $\int_I \rho_1 \, \mathrm{d}m = \int_I \rho_2 \, \mathrm{d}m = 1$, then
  \begin{equation} \label{eq:Lq}
    e^{-(b\lambda_0 + 1) \theta_\mathcal{D} (\rho_1, \rho_2)} \leq \frac{ \int_I \rho_1 \L \phi \, \mathrm{d}m }{ \int_I \rho_2 \L \phi \, \mathrm{d}m } \leq e^{(b\lambda_0 + 1) \theta_\mathcal{D} (\rho_1, \rho_2)}.
  \end{equation}

  We now consider some of the expressions that occur in the expressions for $\alpha_\mathcal{E}$ and $\beta_\mathcal{E}$.
  Let us now choose $b$ so large that $b > b \lambda_0 + 1$. From \eqref{eq:Lq} we get
  \[
    \frac{ e^{b \theta_\mathcal{D} (\rho_1, \rho_2)} - \frac{ \int_I \rho_1 \L \phi_1 \, \mathrm{d}m }{ \int_I \rho_2 \L \phi_1 \, \mathrm{d}m } }{ e^{b \theta_\mathcal{D} (\rho_1, \rho_2)} - \frac{ \int_I \rho_1 \L \phi_2 \, \mathrm{d}m }{ \int_I \rho_2 \L \phi_2 \, \mathrm{d}m } }
    \geq 
    \frac{ e^{b \theta_\mathcal{D} (\rho_1, \rho_2)} - e^{(b\lambda_0 + 1) \theta_\mathcal{D} (\rho_1, \rho_2) } }{ e^{b \theta_\mathcal{D} (\rho_1, \rho_2)} -  e^{-(b\lambda_0 + 1) \theta_\mathcal{D} (\rho_1, \rho_2) } }
  \]
  and
  \[
    \frac{ e^{b \theta_\mathcal{D} (\rho_1, \rho_2)} - \frac{ \int_I \rho_1 \L \phi_1 \, \mathrm{d}m }{ \int_I \rho_2 \L \phi_1 \, \mathrm{d}m } }{ e^{b \theta_\mathcal{D} (\rho_1, \rho_2)} - \frac{ \int_I \rho_1 \L \phi_2 \, \mathrm{d}m }{ \int_I \rho_2 \L \phi_2 \, \mathrm{d}m } }
    \leq 
    \frac{ e^{b \theta_\mathcal{D} (\rho_1, \rho_2)} - e^{-(b\lambda_0 + 1) \theta_\mathcal{D} (\rho_1, \rho_2) } }{ e^{b \theta_\mathcal{D} (\rho_1, \rho_2)} -  e^{(b\lambda_0 + 1) \theta_\mathcal{D} (\rho_1, \rho_2) } }.
  \]
  Similarly as in the proof of Lemma~\ref{lem:contractiononD}, we then get that
  \begin{multline*}
    \frac{b - 1 - b\lambda_0}{b + 1 + b \lambda_0}
    =
    \frac{1 - \frac{1 + b\lambda_0}{b} }{1 + \frac{1 + b\lambda_0}{b} }
    \leq
    \frac{ e^{b \theta_\mathcal{D} (\rho_1, \rho_2)} - \frac{ \int_I \rho_1 \L \phi_1 \, \mathrm{d}m }{ \int_I \rho_2 \L \phi_1 \, \mathrm{d}m } }{ e^{b \theta_\mathcal{D} (\rho_1, \rho_2)} - \frac{ \int_I \rho_1 \L \phi_2 \, \mathrm{d}m }{ \int_I \rho_2 \L \phi_2 \, \mathrm{d}m } } \\
    \leq
    \frac{1 + \frac{1 + b\lambda_0}{b} }{1 - \frac{1 + b\lambda_0}{b} } 
    =
    \frac{b + 1 + b\lambda_0}{b - 1 - b \lambda_0}.
  \end{multline*}
  This shows that
  \[
    \alpha_\mathcal{E} (\L \phi_1, \L \phi_2) \geq     \frac{b - 1 - b\lambda_0}{b + 1 + b \lambda_0} \inf_{\rho \in \D} \frac{ \int \rho \L \phi_2 \, \mathrm{d}m }{ \int \rho \L \phi_1 \, \mathrm{d}m }
  \]
  and
  \[
    \beta_\mathcal{E} (\L \phi_1, \L \phi_2) \leq     \frac{b + 1 + b\lambda_0}{b - 1 - b \lambda_0} \sup_{\rho \in \D} \frac{ \int \rho \L \phi_2 \, \mathrm{d}m }{ \int \rho \L \phi_1 \, \mathrm{d}m }.
  \]
  We therefore have
  \[
    \theta_\mathcal{E} (\L \phi_1, \L \phi_2) \leq \log \Bigl( \frac{b + 1 + b\lambda_0}{b - 1 - b \lambda_0} \Bigr)^2 \sup_{\rho_1, \rho_2 \in \D} 
    \frac{ \int \rho_1 \L \phi_2 \, \mathrm{d}m }{ \int \rho_1 \L \phi_1 \, \mathrm{d}m }
    \frac{ \int \rho_2 \L \phi_1 \, \mathrm{d}m }{ \int \rho_2 \L \phi_2 \, \mathrm{d}m }.
  \]
  It remains to estimate this supremum and show that it is finite.

  Observe first that $1$ is an element in $\D$ with $\int_I 1 \, \mathrm{d}m = 1$. An estimate of the diameter of $\U (\D)$ is given in the proof of Lemma~\ref{lem:contractiononD}. Let $D$ be that number. If $\int_I \rho \, \mathrm{d}m = 1 $, then
  \[
    \int_I \rho \L \phi \, \mathrm{d}m = \int_I \U \rho \phi \, \mathrm{d}m \leq e^{b \theta_\mathcal{D} (1, \U \rho)} \int_I 1 \phi \, \mathrm{d}m \leq e^{bD} \int_I \phi \, \mathrm{d}m
  \]
  and
  \[
    \int_I \rho \L \phi \, \mathrm{d}m = \int_I \U \rho \phi \, \mathrm{d}m \geq e^{- b \theta_\mathcal{D} (1, \U \rho)} \int_I 1 \phi \, \mathrm{d}m \geq e^{-bD} \int_I \phi \, \mathrm{d}m.
  \]
  This implies that
  \[
    e^{-2bD} \leq \frac{ \int_I \rho_1 \L \phi \, \mathrm{d}m }{ \int_I \rho_2 \L \phi \, \mathrm{d}m } \leq e^{2bD},
  \]
  and so
  \[
    \sup_{\rho_1, \rho_2 \in \D} 
    \frac{ \int \rho_1 \L \phi_2 \, \mathrm{d}m }{ \int \rho_1 \L \phi_1 \, \mathrm{d}m }
    \frac{ \int \rho_2 \L \phi_1 \, \mathrm{d}m }{ \int \rho_2 \L \phi_2 \, \mathrm{d}m } \leq e^{4bD}.
  \]
  Finally, we get that
  \begin{multline*}
    \theta_\mathcal{E} ( \L (\E) ) \leq 4bD + 2 \log \Bigl( \frac{b + 1 + b\lambda_0}{b - 1 - b \lambda_0} \Bigr) \\
    = 8b \log \frac{1 + \lambda}{1 - \lambda} + 2^{3+\gamma} \lambda a b + 2 \log \biggl( \frac{b + 1 + b \tanh \bigl( \frac{1}{2} \log \frac{1 + \lambda}{1 - \lambda} + 2^{\gamma-1} \lambda a \bigr)}{b - 1 - b \tanh \bigl( \frac{1}{2} \log \frac{1 + \lambda}{1 - \lambda} + 2^{\gamma-1} \lambda a \bigr)} \biggr),
  \end{multline*}
  where we used that $D = 2 \log \frac{1 + \lambda}{1 - \lambda} + 2^{1+\gamma} \lambda a$ and $\lambda_0 = \tanh \bigl( \frac{1}{2} \log \frac{1 + \lambda}{1 - \lambda} + 2^{\gamma-1} \lambda a \bigr)$.
  The condition $b > b \lambda_0 + 1$ is equivalent to $b > \Bigl(1 - \tanh \bigl( \frac{1}{2} \log \frac{1 + \lambda}{1 - \lambda} + 2^{\gamma-1} \lambda a \bigr) \Bigr)^{-1}$.
\end{proof}

Lemma~\ref{lem:finiteEdiam} together with Theorem~\ref{the:contraction}, shows that $\L$ is a contraction in the cone metric. We get the following corollary.

\begin{corollary} \label{cor:Liscontraction}
  There is a number $\lambda_1 < 1$ such that $\theta_\mathcal{E} (\L (\phi_1), \L (\phi_2)) \leq \lambda_1 \theta_\mathcal{E} (\phi_1, \phi_2)$ holds for all $\phi_1, \phi_2 \in \E$.
\end{corollary}

\section{Invariant measure}

We will use Corollary~\ref{cor:Liscontraction} to construct the invariant measure of the iterated function system. To do this we will need the following Lemma.

\begin{lemma} \label{lem:cauchy}
  If $(\phi_n)_{n \in \mathbb{N}}$ is a Cauchy sequence in the projective metric $\theta_\mathcal{E}$ of $\E$, such that $\int_I \phi_n \, \mathrm{d}m = 1$, and $\psi$ is any continuous function on $I$, then the sequence $\bigl( \int_I \psi \phi_n \, \mathrm{d}m \bigr)_{n \in \mathbb{N}}$ is Cauchy.
\end{lemma}

\begin{proof}
  Let first $\psi$ be a function in $\D$. Then
  \[
    \biggl| \int_I \psi \phi_m \, \mathrm{d}m - \int_I \psi \phi_n \, \mathrm{d}m \biggr| \leq \sup \psi \Biggl| \frac{ \int_I \psi \phi_m \, \mathrm{d}m }{ \int_I \psi \phi_n \, \mathrm{d}m } - 1\Biggr| \leq \sup \psi \bigl| e^{ \theta_\mathcal{E} (\phi_m, \phi_n)} - 1 \bigr|.
  \]
  As $\theta_\mathcal{E} (\phi_m, \phi_n) \to 0$ when $m,n \to \infty$, this shows that $\int_I \psi \phi_n \, \mathrm{d}m$ is Cauchy.

  Now, let $\psi$ be $\gamma$-H\"older. Then the function $\psi + c$ is in $\D$ if $c$ is sufficiently large. By what we just proved, it follows that $\int_I (\psi + c) \phi_n \, \mathrm{d}m$ is Cauchy. Since $\int_I c \phi_n \, \mathrm{d}m = c$ for all $n$, it follows that $\int_I \psi \phi_n \, \mathrm{d}m$ is Cauchy.

  Let $\varepsilon > 0$. For any continuous function $\psi$, we may take a $\gamma$-H\"older continuous function $\psi_0$ such that $\sup |\psi - \psi_0| < \varepsilon$. Then
  \[
    \biggl| \int_I \psi \phi_m \, \mathrm{d}m - \int_I \psi \phi_n \, \mathrm{d}m \biggr|
    \leq
    \biggl| \int_I \psi_0 \phi_m \, \mathrm{d}m - \int_I \psi_0 \phi_n \, \mathrm{d}m \biggr| + 2 \varepsilon < 3 \varepsilon
  \]
  if $n$ and $m$ are sufficiently large. Hence $\int_I \psi \phi_n \, \mathrm{d}m$ is Cauchy.
\end{proof}

We now let $\phi_n = \L^n 1$. This is a Cauchy sequence in the metric $\theta_\mathcal{E}$ according to Corollary~\ref{cor:Liscontraction}.

Moreover
\[
  \int_I \phi_1 \, \mathrm{d}m = \int_I 1 \cdot \L^n 1 \, \mathrm{d}m = \int_I \U 1 \cdot 1 \, \mathrm{d}m = \int_I 1 \, \mathrm{d}m = 1.
\]
Hence $\mu_n (\psi) = \int_I \psi \phi_n \, \mathrm{d}m$ defines a probability measure. We let $\mu$ be the weak limit of $\mu_n$ as $n \to \infty$. This limit exists by Lemma~\ref{lem:cauchy}.

\section{Smooth density}
We are now ready to prove that the measure $\mu$ is absolutely continuous with respect to Lebesgue measure, with density in $\mathcal{C}^\infty$. So far, our results are valid for $\varepsilon \geq 0$. We now assume that $\varepsilon > 0$.

If $\psi$ if a differentiable function then
\[
  \frac{\mathrm{d}}{\mathrm{d}x} \U \psi (x) = \lambda \U (\psi') (x).
\]
Note that the corresponding formula for $\frac{\mathrm{d}}{\mathrm{d}x} \L \psi (x)$ is  more involved because of the presens of $\chi_{f_{i,t} (I)}$ in \eqref{eq:Ldef}. However if $\psi \in \mathcal{C}_0^1 (I)$, then we have
\[
  \frac{\mathrm{d}}{\mathrm{d}x} \L \psi (x) = \frac{1}{\lambda} \L (\psi') (x).
\]

We then get that
\begin{multline*}
  \int_I \psi \cdot (\L^n \phi)^{(k)} \, \mathrm{d}m
  = \frac{1}{\lambda^{k l}} \int_I \psi \cdot \L^l (\L^{n-l} \phi)^{(k)} \, \mathrm{d}m \\
  = \frac{1}{\lambda^{k l}} \int_I \U^l \psi \cdot (\L^{n-l} \phi)^{(k)} \, \mathrm{d}m
  = \frac{(-1)^k}{\lambda^{k l} } \int_I (\U^l \psi)^{(k)} \cdot \L^{n-l} \phi \, \mathrm{d} m,
\end{multline*}
holds for any $\phi$, $\psi$ and $k$, $l$ such that the derivatives exist.

Let $t \in I$ and $\psi_t = \chi_{[-1, t]}$. Then
\begin{multline} \label{eq:derivative}
  \phi_n^{(k-1)} (t) - \phi_n^{(k-1)} (0) 
  = \phi_n^{(k-1)} (t) = \int_{-1}^t \phi_n^{(k)} \, \mathrm{d}m \\
  = \int_I \psi_t \cdot (\L^n 1)^{(k)} \, \mathrm{d}m 
  = \frac{(-1)^k}{\lambda^{k l}} \int_I (\U^l \psi_t)^{(k)} \cdot \L^{n-l} 1 \, \mathrm{d}m.
\end{multline}
If we let $l = k + 2$, then $(\U^l \psi_t)^{(k)}$ is defined, and the rightmost integral in \eqref{eq:derivative} converges by Lemma~\ref{lem:cauchy}. This shows that the density of $\mu$ is in $\mathcal{C}^\infty$. Moreover we get
\begin{multline*}
  | \phi_n^{(k-1)} (t) | \leq \biggl| \int_{-1}^t \phi_n^{(k)} \, \mathrm{d}m \biggr| = \biggl|  \frac{1}{\lambda^{k(k+2)}} \int_I (\U^{k+2} \psi_t)^{(k)} \cdot \L^{n-k-2} 1 \, \mathrm{d}m \biggr| \\
   \leq \frac{ \sup | (\U^{k+2} \psi_t)^{(k)} | }{\lambda^{k(k+2)}} \int_I \L^{n-k-2} 1 \, \mathrm{d}m
   = \frac{ \sup | (\U^{k+2} \psi_t)^{(k)} | }{\lambda^{k(k+2)}}.
\end{multline*}

Let us estimate $\sup | (\U^{k+2} \psi_t)^{(k)} |$.
We first observe that if $\psi$ is continuous then
\begin{align*}
  \frac{\mathrm{d}}{\mathrm{d}x} \U \psi (x)
  &= \frac{\mathrm{d}}{\mathrm{d}x} \sum_{i = 1}^n p_i \int_0^\varepsilon \psi \circ f_{i, t} (x) h(t) \, \mathrm{d}m \\
  &= \frac{\mathrm{d}}{\mathrm{d}x} \sum_{i = 1}^n p_i \int_{\lambda x + a_i}^{\lambda x + a_i + b_i \varepsilon} \psi (s) h \Bigl( \frac{ s - \lambda x - a_i }{ b_i } \Bigr) \frac{1}{b_i} \, \mathrm{d}m \\
  &= \sum_{i = 1}^n \frac{p_i \lambda}{b_i} \bigl( \psi \circ f_{i, \varepsilon} (x) h (\varepsilon) - \psi \circ f_{i, 0} (x) h (0) \bigr) \\
  &\hspace{2cm} - \sum_{i = 1}^n \frac{p_i \lambda}{b_i} \int_0^\varepsilon \psi \circ f_{i, t} (x) h'(t) \, \mathrm{d}t.
\end{align*}
This implies that
\[
  \sup \Bigl| \frac{\mathrm{d}}{\mathrm{d}x} \U \psi (x) \Bigr|
  \leq \sum_{i = 1}^n \frac{p_i \lambda}{b_i} \sup |\psi| (h(\varepsilon) + h(0) + \varepsilon \sup |h'|).
\]
Writing $(\U^{k+2} \psi_t)^{(k)}$ as
\[
  (\U^{k+2} \psi_t)^{(k)} (x) = \lambda^{\frac{k (k + 1)}{2} } \U \frac{\mathrm{d}}{\mathrm{d}x} \U \frac{\mathrm{d}}{\mathrm{d}x} \U \cdots \frac{\mathrm{d}}{\mathrm{d}x} \U^2 \psi_t (x),
\]
we can conclude that
\begin{multline*}
  \sup | \U^{k+2} \psi_t)^{(k)} | \leq \lambda^{\frac{k (k + 1)}{2} } \biggl( \sum_{i = 1}^n \frac{p_i \lambda}{b_i} (h(\varepsilon) + h(0) + \varepsilon \sup |h'|) \biggr)^k \sup |\U^2 \psi| \\
  \leq
  \lambda^{\frac{k (k + 1)}{2} } \biggl( \sum_{i = 1}^n \frac{p_i \lambda}{b_i} (h(\varepsilon) + h(0) + \varepsilon \sup |h'|) \biggr)^k.
\end{multline*}
This yields
\[
  \sup | \phi_n^{(k-1)} | \leq \lambda^{-\frac{k(k+1)}{2} } \biggl( \sum_{i = 1}^n \frac{p_i}{b_i} (h(\varepsilon) + h(0) + \varepsilon \sup |h'|) \biggr)^k.
\]
If we let $\phi$ denote the density of the measure $\mu$, then
\[
  \sup | \phi^{(k)} | \leq \lambda^{ - \frac{(k+1)(k+2) }{2} } \biggl( \sum_{i = 1}^n \frac{p_i}{b_i} (h(\varepsilon) + h(0) + \varepsilon \sup |h'|) \biggr)^{k+1}.
\]

\end{document}